\documentclass[reqno, 11pt]{amsart}
\usepackage{amssymb,amsfonts, amsmath}
\usepackage{mathrsfs}
\usepackage{color}
\usepackage{dsfont}
\usepackage[normalem]{ulem} % \sout{} for strike fonts
\newtheorem{theorem}{Theorem}[section]

\newtheorem{lemma}{Lemma}[section]

\numberwithin{equation}{section}
\newcommand{\E}{\mathsf{E}}
\renewcommand{\b}[1]{\mbox{\boldmath $#1$}}

\theoremstyle{definition}

\definecolor{gray}{rgb}{0.9,0.9,0.9}

\input cyracc.def

\begin{document}

\title[] {Bene$\check{\bf S}$ condition for discontinuous
exponential martingale}
% ----------------------------------------------------------------
\author{R. Liptser}
\address{Department of Electrical Engineering Systems,
Tel Aviv University, 69978 Tel Aviv, Israel}
\email{liptser@eng.tau.ac.il; rliptser@gmail.com}


%\commby{}%

\maketitle
\begin{abstract}
It is known the Girsanov exponent $\mathfrak{z}_t$, being solution
of Doleans-Dade equation $
\mathfrak{z_t}=1+\int_0^t\alpha(\omega,s)dB_s $ generated by
Brownian motion $B_t$ and a random process $\alpha(\omega,t)$ with
$\int_0^t\alpha^2(\omega,s)ds<\infty$ a.s., is the martingale
provided that the Bene${\rm \check{s}}$ condition
$$
|\alpha(\omega,t)|^2\le \text{\rm
const.}\big[1+\sup_{s\in[0,t]}B^2_s\big], \ \forall \ t>0,
$$
holds true. In this paper, we show $B_t$ can be replaced by by a
homogeneous purely discontinuous square integrable martingale $M_t$
with independent increments and paths from the Skorokhod space $
\mathbb{D}_{[0,\infty)} $ having positive jumps $\triangle M_t$ with
$\E\sum_{s\in[0,t]}(\triangle M_s)^3<\infty$. A function
$\alpha(\omega,t)$ is assumed to be nonnegative and predictable.
Under this setting $\mathfrak{z}_t$ is the martingale provided that
$$
\alpha^2(\omega,t)\le \text{\rm
const.}\big[1+\sup_{s\in[0,t]}M^2_{s-}\big], \ \forall \ t>0.
$$
The method of proof differs from  the original  Bene${\rm
\check{s}}$ one and is compatible for both setting with $B_t$ and
$M_t$.

\end{abstract}

\section{\bf Introduction and main result}

\subsection{Setting of problem}A classical Girsanov's exponent
\begin{equation*}
\mathfrak{z}_t=\exp\Big(\int_0^t\alpha(\omega,s)dB_s-\frac{1}{2}\int_0^t\alpha^2(\omega,s)ds
\Big),
\end{equation*}
with Brownian motion $B_t$ and adapted random process
$\alpha(\omega,t)$ having $\int_0^t\alpha(\omega,s)ds<\infty$, forms
a positive local martingale (and supermartingale too) with
$\E\mathfrak{z}_t\le 1$. If
\begin{equation}\label{eq:1+1+1}
\E\mathfrak{z}_t\equiv 1,
\end{equation}
the random process $\mathfrak{z}_t$ is a martingale. In order
\eqref{eq:1+1+1} to have, Girsanov in \cite{Girs} used bounded
function $\alpha(\omega,t)$  and suggested a conjecture that
\eqref{eq:1+1+1} will be valid if $\alpha^2(\omega,t)\approx B^2_t$.
Among  conditions guaranteing \eqref{eq:1+1+1} (see, e.g. Novikov
\cite{12}, Kazamaki \cite{7}, the latest Krylov \cite{Kry}, etc), we
distinguish the Bene${\rm \check{s}}$ statement: for any $T>0$,
\begin{equation}\label{eq:Benes}
``|\alpha(\omega,t)|^2\le \text{\rm
const.}\big[1+\sup_{s\in[0,t]}B^2_s\big]_{t\in[0,T]}\text{''}\Rightarrow
``\E\mathfrak{z}_T=1\text{''},
\end{equation}
which is derived in \cite{Benes} with the help of Kazamaki \cite{7}
(see also Karatzas and Shreve \cite{142}, "Ust\"unel and Zakai
\cite{UZ}).

The aim of this paper is to ``replace'' the Brownian motion $B_t$ by
a homogeneous purely discontinuous square integrable martingale
$M_t$ with independent increments and obtain an implication similar
to \eqref{eq:Benes}. Unfortunately, a method of proof of
\eqref{eq:Benes} with $M_t$ instead of $B_t$ is not applicable since, for example, results
of Kazamaki, Novikov, and  Krylov related to martingale $M_t$ do not exist.

\subsection{New approach to Bene${\bf \check{s}}$  result}\label{sec-1.2}
We propose a new  approach for the proof of \eqref{eq:Benes} which
is compatible with $B_t$ and $M_t$.
Its application to a discontinuous martingale is more involved than to classical case
with Brownian motion. So, to make our approach at most transparent we give
a sketch of the proof of \eqref{eq:Benes} step by step.

{\bf (1)} It is
well known $ \mathfrak{z}_t $ is the unique solution of
Dol\'eans-Dade equation

\begin{equation}\label{eq:DDz+}
\mathfrak{z}_t=1+\int_0^t\mathfrak{z}_s\alpha(\omega,s)dB_s.
\end{equation}
Set $\sigma_n=\inf\Big\{t:\Big[1+\sup_{s\in[0,t]}B^2_s\Big]\ge
n\Big\}$, $B^n_t=B_{t\wedge \sigma_n}$ and
$\mathfrak{z}^n_t=\mathfrak{z}_{t\wedge\sigma_n}$. Then
$$
\mathfrak{z}^n_t=1+\int_0^t\mathfrak{z}^n_sI_{\{\sigma_n\ge
s\}}\alpha(\omega,s)dB_s.
$$
Evidently,
$
I_{\{\sigma_n\ge s\}}\alpha^2(\omega,s)\le \text{const.}n. $
Consequently, $ \E\mathfrak{z}^n_t\equiv 1. $

{\bf (2)} In order to prove
$\E\mathfrak{z}_t=1$ for any $t\in[0,T]$, it suffices to show the
family $\{\mathfrak{z}^n_T\}_{n\to\infty}$ is uniformly integrable.
With chosen $\sigma_n$, we have $\E\mathfrak{z}^n_T=1$. Let us
introduce a probability measure
$\widetilde{\mathsf{P}}^n_T\ll\mathsf{P}$ with $
d\widetilde{\mathsf{P}}^n_T=\mathfrak{z}^n_Td\mathsf{P} $ and denote
by $\widetilde{\E}^n_T$ the expectation symbol of $
\widetilde{\mathsf{P}}^n_T. $

Following Hitsuda \cite{Hits}, the uniform integrability of
$\{\mathfrak{z}^n_T\}_{n\to\infty}$ is verified with a convex
function
\begin{equation}\label{eq:psiips}
\psi(x)=x\log(x)+1-x, \ x\ge 0
\end{equation}
due to the Vall\'ee-Poussin's criteria since
$\lim_{x\to\infty}\frac{\psi(x)}{x}=\infty.$
Namely, we have to show that $\sup_n\E\psi(\mathfrak{z}^n_T)<\infty.$
A verification of this condition is inconvenient. However if
$\E\psi(\mathfrak{z}^n_T)<\infty$
a direct computation shows
$
\E\psi(\mathfrak{z}^n_T)=\widetilde{\E}^n_T\log\big(\mathfrak{z}^n_T\big).
$

{\bf (3)} the Girsanov theorem, a random process
$(\widetilde{B}^n_t)_{t\in[0,T]}$ with

\begin{equation}\label{eq:BBB}
\widetilde{B}^n_t=B^n_t-\int_0^tI_{\{\sigma_n\ge
s\}}\alpha(\omega,s)ds
\end{equation}
is $\widetilde{\mathsf{P}}^n_T$-martingale with the predictable
quadratic variation $ \langle \widetilde{B}^n\rangle_t= \langle
B^n\rangle_t\equiv t\wedge\sigma_n.$
Therefore, by using \eqref{eq:BBB}, we obtain
$$
\log(\mathfrak{z}^n_t)\stackrel{\widetilde{\mathsf{P}}^n_T}{=}\int_0^tI_{\{\sigma_n\ge
s\}}\alpha(\omega,s) d\widetilde{B}^n_s +\frac{1}{2}
\int_0^tI_{\{\sigma_n\ge s\}}\alpha^2(\omega,s)ds.
$$
Since
$
I_{\{\sigma_n\ge s\}}\alpha^2(\omega,s)
$
is bounded, we have
$
\widetilde{\E}^n_T\int_0^TI_{\{\sigma_n\ge s\}}\alpha^2(\omega,s)ds\le c_n
$
with a constant $c_n$ depending on $n$. Hence,
$
\widetilde{\E}^n_T\int_0^tI_{\{\sigma_n\ge
s\}}\alpha(\omega,s) d\widetilde{B}^n_s=0.
$

Therefore and in view of
$
I_{\{\sigma_n\ge s\}}\alpha^2(\omega,s)ds\le \text{const.}[1+\sup_{s'\in[0,s\wedge\sigma_n]}
|B^n_{s'}|^2],
$

$$
\widetilde{\E}^n_T\log(\mathfrak{z}^n_T)=\widetilde{\E}^n_T\int_0^TI_{\{\sigma_n\ge s\}}
\alpha^2(\omega,s)ds
\\
\le \text{const.}\Big[1+\widetilde{\E}^n_T \sup_{s'\in[0,T\wedge\sigma_n]}
|B^n_{s'}|^2\Big].
$$

{\bf (4)} Now, the proof is reduced to
$
\sup_n\widetilde{\E}^n_T \sup_{s'\in[0,T\wedge\sigma_n]}|B^n_{s'}|^2<\infty.
$

Denote $V^n_t:=\widetilde{\E}^n_T \sup_{s'\in[0,t\wedge\sigma_n]}
|B^n_{s'}|^2$. In view of \eqref{eq:BBB},
\begin{gather*}
V^n_t\le 2\widetilde{\E}^n_T\Big(\int_0^tI_{\{\sigma_n\ge
s\}}|\alpha(\omega,s)|ds\Big)^2+2\widetilde{\E}^n_T\Big(\sup_{s'\in[0,t\wedge\sigma_n]}
|\widetilde{B}^n_{s'}|\Big)^2.
\end{gather*}
By the Doob maximal inequality,
$
\widetilde{\E}^n_T\Big(\sup\limits_{s'\in[0,t\wedge\sigma_n]}
|\widetilde{B}^n_{s'}|\Big)^2\le 4\widetilde{\E}^n_T|\widetilde{B}^n_{t\wedge\sigma_n}|^2
=4\widetilde{\E}^n_T|(t\wedge\sigma_n)\le 4T,
$
while by the Cauchy-Schwarz inequality
$$
\widetilde{\E}^n_T\Big(\int_0^tI_{\{\sigma_n\ge
s\}}|\alpha(\omega,s)|ds\Big)^2\le \widetilde{\E}^n_T\int_0^tI_{\{\sigma_n\ge
s\}}\alpha^2(\omega,s)ds\le\mathbf{r}\Big[1+\int_0^tV^n_s\Big].
$$
Finally, combining these estimates,  we obtain an integral inequality (with the constant
$\mathbf{r}$ independent of $n$):
$$
V^n_t\le \mathbf{r}\Big[1+\int_0^tV^n_sds\Big].
$$

Thus, $V^n_T\le \mathbf{r}e^{T\mathbf{r}}$.

\subsection{Formulation of main result}

Let us explain how Brownian $B_t$ might be replaced by a purely
discontinuous martingale $M_t$. A simplest way is to replace $B_t$
by $M_t$ in the the Dol\'eans-Dade equation \eqref{eq:DDz+}, that
is,
\begin{equation*}
\mathfrak{z}_t=1+\int_0^t\mathfrak{z}_{s-}\alpha(\omega,s)dM_s,
\end{equation*}
where $\mathfrak{z}_{s-}=\lim\limits_{s'\uparrow s}\mathfrak{z}_s'$,
and adapted process $\alpha(\omega,t)$ is replaced by its
predictable version. The square integrable  martingale $M_t$ has
paths the Skorokhod space $ \mathbb{D}_{[0,\infty)}. $ Denote
$\langle M\rangle_t$ the predictable quadratic variation of $M_t$
and $M_{t-}= \lim_{t'\uparrow t}M_{t'}$.

We assume $\int_0^t\alpha^2(\omega,s)d\langle M\rangle_s<\infty$. A
positiveness of $\mathfrak{z}_t$ is warranted by assumptions
$\alpha(\omega,t)\ge 0$ and $ \big(M_t-M_{t-}\big) I_{\{M_t\ne
M_{t-}\}}>0. $

 We choose $M_t$
in a form of It\^o's integral

\begin{equation*}
M_t=\int_0^t\int_{\mathbb{R}_+}z\big[\mu(ds,dz)-\nu(ds,dz)\big]
\end{equation*}
relative to, so called, martingale difference ``$\mu-\nu$'', where
$\mu(dt,dz)$ is the integer-valued measure $\mu=\mu(dt,dz)$
associated with a jump process $\triangle M_t=M_t-M_{t-}$ of $M_t$
and $\nu(dt,dz)$ is a compensator of $\mu(dt,dz)$. In order to have
the above-mentioned properties of $M_t$, we choose a deterministic
compensator
\begin{equation*}
\nu(dt,dz)=K(dz)dt
\end{equation*}
with a measure $K(dz)$ supported on $\mathbb{R}_+$ and
$\int_{\mathbb{R}_+}z^2K(dz)<\infty$. In particular, then,
$$
\langle M\rangle_t\equiv\E M^2_t\equiv\int_0^tz^2K(dz)ds.
$$

Our main result is formulated in
\begin{theorem}\label{theo-00}
Assume $\int_{\mathbb{R}_+}z^3K(dz)<\infty$. Then for any $T>0$ {\rm (}comp.
\eqref{eq:Benes}{\rm )}
\begin{equation*}
``|\alpha(\omega,t)|^2\le \text{\rm
const.}\big[1+\sup_{s\in[0,t]}M^2_{s-}\big]_{t\in[0,T]}\text{''}\Rightarrow
``\E\mathfrak{z}_T=1\text{''},
\end{equation*}

\end{theorem}

The method of proof is similar to one given in Section
\eqref{sec-1.2}

\section{\bf The proof of Theorem \ref{theo-00}}
\label{sec-2.z}

\subsection{Preliminaries}

We begin  with recalling  necessary notions (for more
details, see e.g. \cite{LSMar} or \cite{JSn}). Along this paper a
filtered probability space
 $(\varOmega, \mathcal{F},
(\mathscr{F}_t)_{t\in[0,\infty)},\mathsf{P})$ with ``general
conditions'' is fixed and all random objects are defined on it.
$\mathscr{P}$ denotes predictable $\sigma$-algebra relative to
$(\mathscr{F}_t)_{t\in[0,\infty)}$ and $\mathscr{B}_+$ is the Borel
$\sigma$-algebra on $\mathbb{R}_+$.

Henceforth, $\mathbf{r}$ denotes a generic constant taking different
values at different appearances and is independent of a number $n$ involved in the text.

We begin with know implication:
\begin{equation}\label{eq:=1}
 ``\alpha(\omega,t)\le \mathbf{r}\text{''}\Rightarrow
``\E\mathfrak{z}_t\equiv 1 \text{''}.
\end{equation}

Set $\tau_n=\inf\{t:\mathfrak{z}_{t-} \ge n\}$ and notice
$\mathfrak{z}_{(t\wedge\tau_n)-}\le n$. Then $\mathfrak{z}^n_t:=
\mathfrak{z}_{t\wedge\tau_n}$ solves the Doleans-Dade equation
\begin{equation}\label{eq:3.1d}
\mathfrak{z}^n_t=1+\int_0^t\mathfrak{z}^n_{s-}I_{\{\tau_n\ge
s\}}\alpha(\omega,s)dM_s.
\end{equation}
A boundedness of $I_{\{\tau_n\ge
s\}}\mathfrak{z}^n_{s-}\alpha(\omega,s)$ guarantees the process
$\mathfrak{z}^n_t$ is the square integrable martingale and $
\E(\mathfrak{z}^n_t)^2= 1+\E\int_0^t\big(I_{\{\tau_n\ge
s\}}\mathfrak{z}^n_{s-}\alpha(\omega,s)\big)^2d\langle M\rangle_sds
\le 1+\mathbf{r}\int_0^t\E(\mathfrak{z}^n_s)^2ds. $ Then, a function
$ V^n_t=\E(\mathfrak{z}^n_t)^2 $ solves the integral inequality: $
V^n_t\le 1+\mathbf{r}\int_0^tV^n_sds. $ So, by the Bellman-Gronwall
inequality, $V^n_t\le e^{\mathbf{r}t}$, that is, $
\sup_n\E(\mathfrak{z}^n_t)^2\le e^{\mathbf{r}t}. $ Therefore, by the
Vall\'ee-Poussin's criteria, the family
\{$\mathfrak{z}^n_t)\}_{n\to\infty}$ is uniformly integrable. So,
not only $ \lim_{n\to\infty}\mathfrak{z}^n_t=\mathfrak{z}_t $ but
also $ \E\mathfrak{z}_t=\lim_{n\to\infty}\E\mathfrak{z}^n_t\equiv 1.
$

\subsection{$\b{\mathfrak{z}^n_t}$ approximation of $\b{ \mathfrak{z}_t}$.
Change of probability measure}

\begin{lemma}\label{lem-2.2z}
Let $\sigma_n=\inf\big\{t:\big[1+\sup_{s\in[0,t]}M^2_{s-}\big]\ge
n\big\}$ and $\mathfrak{z}^n_t=\mathfrak{z}_{t\wedge\sigma_n}$. Then
$\E\mathfrak{z}^n_t\equiv 1$.
\end{lemma}
\begin{proof}
We use \eqref{eq:3.1d} with $\tau_n$ replaced by $\sigma_n$:
\begin{equation}\label{eq:sigman}
\mathfrak{z}^n_t=1+\int_0^tI_{\{\sigma_n\ge
s\}}\mathfrak{z}^n_{s-}\alpha(\omega,s)dM_s.
\end{equation}
Since
$
\big[1+\sup\limits_{s\in[0,t\wedge\sigma_n]}M^2_{s-}\big]\le n
$
and, then,
$ I_{\{\sigma_n\ge t\}}|\alpha(\omega,t)|^2\le
\mathbf{r}I_{\{\sigma_n\ge
t\}}\big[1+\sup\limits_{s\in[0,t]}M^2_{s-}\big]\le\mathbf{r}n$,
it remains to apply \eqref{eq:=1}.
\end{proof}

Let $T>0$ be fixed.
By Lemma \ref{lem-2.2z} $\E\mathfrak{z}^n_T=1.$
As in Section \ref{sec-1.2}, we have to show the family
$\{\mathfrak{z}^n_T\}_{n\to\infty}$ is uniformly  integrable. So, we intend to
to verify $\sup_n\E\psi(\mathfrak{z}^n_t)<\infty $ with the
function $\psi(x)$ defined in \eqref{eq:psiips}.

Repeating arguments from Section \ref{sec-1.2}, we introduce
a probability measure $\widetilde{\mathsf{P}}^n_T\ll \mathsf{P}$
with
$d\widetilde{\mathsf{P}}^n_T=\mathfrak{z}^n_Td\mathsf{P}$ ($\widetilde{\E}^n_T$ denotes
the expectation
symbol of $\widetilde{\mathsf{P}}^n_T$).
Set $M^n_t=M_{t\wedge\sigma_n}$ and rewrite \eqref{eq:sigman} in equivalent form
\begin{equation}\label{eq:sigmanz}
\mathfrak{z}^n_t=1+\int_0^tI_{\{\sigma_n\ge
s\}}\mathfrak{z}^n_{s-}\alpha(\omega,s)dM^n_s.
\end{equation}
The random process $(M^n_t)_{t\in[0,T]}$ is
$\mathsf{P}$ - square integrable martingale. Since
$\widetilde{\mathsf{P}}^n_T\ll\mathsf{P}$, the process
$(M^n_t)_{t\in[0,T]}$ is $\widetilde{\mathsf{P}}^n_T$-semimartingale
obeying the unique decomposition $ M^n_t=A^n_t+\widetilde{M}^n_t $ with predictable
drift $A^n_t$ and local martingale $\widetilde{M}^n_t$ (see, e.g.
\cite{LSMar}, Ch.4, \S5, Theorem 2). Denote
$\widetilde{\nu}^n_T(dt,dz)$ a compensator of $\mu(dt,dz)$ relative
to $\widetilde{\mathsf{P}}^n_T$.

\begin{lemma}\label{lem-2.4z0}

\mbox{}

{\bf 1.}
$
\widetilde{\nu}^n_T(ds,dz)=I_{\{\sigma_n\ge s\}}
\big(1+\alpha(\omega,s)z)K(dz)ds
$

{\bf 2.}
$
\widetilde{M}^n_t=M^n_t-A^n_t
$
is square integrable martingale with

\begin{equation}\label{eq:uzhe}
\langle \widetilde{M}^n\rangle_t=\int_0^t\int_{\mathbb{R}_+}z^2I_{\{\sigma_n\ge s\}}
\big(1+\alpha(\omega,s)z)K(dz)ds
\end{equation}

{\bf 3.}
$
A^n_t=\int_0^t\int_{\mathbb{R}_+}I_{\{\sigma_n\ge s\}}\alpha(\omega,s)z^2K(dz)ds
$
\end{lemma}
\begin{proof} {\bf 1.}
Below, we will use a formula
$$
\frac{\mathfrak{z}^n_{s}}{\mathfrak{z}^n_{s-}}I_{\{\sigma_n\ge s\}}
=I_{\{\sigma_n\ge s\}}\big(1+\alpha(\omega,s)\triangle
M^n_s)
$$
readily derived from \eqref{eq:sigmanz}.
Let $u(\omega,t,z)$ be bounded and
$\mathscr{P}\otimes\mathscr{B}(\mathbb{R}_+)$-measurable function
vanishing in a vicinity of $\{0\}$. Write
\begin{align*}
&
\widetilde{\E}^n_T\int_0^T\int_{\mathbb{R}_+}I_{\{\sigma_n\ge
s\}}u(\omega,s,z)\mu(dt,dz)
\\
&=\E\mathfrak{z}^n_T\int_0^T\int_{\mathbb{R}_+}I_{\{\sigma_n\ge
s\}}u(\omega,s,z)\mu(dt,dz)
\\
&=\E\int_0^T\int_{\mathbb{R}_+}\mathfrak{z}^n_sI_{\{\sigma_n\ge
s\}}u(\omega,s,z)\mu(dt,dz)
\\
&=\E\int_0^T\int_{\mathbb{R}_+}\mathfrak{z}^n_{s-}
\frac{\mathfrak{z}^n_{s}}{\mathfrak{z}^n_{s-}}I_{\{\sigma_n\ge
s\}}u(\omega,s,z)\mu(dt,dz)
\\
&=\E\int_0^T\int_{\mathbb{R}_+}\mathfrak{z}^n_{s-}
[1+\alpha(\omega,s)\triangle M^n_s]I_{\{\sigma_n\ge
s\}}u(\omega,s,z)\mu(dt,dz)
\\
&=\E\int_0^T\int_{\mathbb{R}_+}\mathfrak{z}^n_{s-}I_{\{\sigma_n\ge
s\}}\big[1+\alpha(\omega,s)z]u(\omega,s,z)\mu(ds,dz)
\\
&=\E\int_0^T\int_{\mathbb{R}_+}\mathfrak{z}^n_{s-}I_{\{\sigma_n\ge
s\}} \big[1+\alpha(\omega,s)z]u(\omega,s,z)K(dz)ds
\\
&=\E\mathfrak{z}^n_T\int_0^T\int_{\mathbb{R}_+}I_{\{\sigma_n\ge s\}}
\big[1+\alpha(\omega,s)z]u(\omega,s,z)K(dz)ds
\\
&=\widetilde{\E}^n_T\int_0^T\int_{\mathbb{R}_+}I_{\{\sigma_n\ge s\}}
\big[1+\alpha(\omega,s)z]u(\omega,s,z)K(dz)ds.
\end{align*}

The desired  result follows by arbitrariness of $u(\omega,t,z)$ .

{\bf 2.+3.} Recall $\int_{\mathbb{R}_+}(z^2\vee z^3)K(dz)<\infty$. Hence
$
\widetilde{\E}^n_T\int_0^T\int_{\mathbb{R}_+}z^2\widetilde{\nu}^n(ds,dz)<\infty.
$
Therefore
$$
\widehat{M}^n_t=\int_0^t\int_{\mathbb{R}_+}zI_{\{\sigma_n\ge s\}}
[\mu(ds,dz)-\widetilde{\nu}^n_T(ds,dz)
$$
is $\widetilde{\mathsf{P}}^n_T$-square integrable martingale with
predictable quadratic variation process

\begin{equation}\label{eq:uhUh}
\langle \widehat{M}^n\rangle_t=\int_0^t\int_{\mathbb{R}_+}z^2I_{\{\sigma_n\ge s\}}
\big(1+\alpha(\omega,s)z)K(dz)ds.
\end{equation}

On the other hand,
 \begin{gather*}
M^n_t-\widehat{M}^n_t=\int_0^t\int_{\mathbb{R}_+}I_{\{\sigma_n\ge s\}}
z[\widetilde{\nu}^n_T(ds,dz)-K(dz)]ds
 \\
 =\int_0^t\int_{\mathbb{R}_+}I_{\{\sigma_n\ge s\}}z^2\alpha(\omega,s)K(dz)ds=:A^n_t,
 \end{gather*}
where $A^n_t$ is well defined predictable process.

Consequently, $\widetilde{M}^n_t\equiv \widehat{M}^n_t$ in view of the unique
semimartingale decomposition with the predictable drift.
\end{proof}

\subsection{Upper bound of $\b{\E\psi(\mathfrak{z}^n_T)}$}

Following the main idea of Section \ref{sec-1.2}, we have to show the family
$\{\mathfrak{z}^n_T\}_{n\to\infty}$ is uniformly integrable.
To this end, we have to prove
$
\sup_n\E\psi(\mathfrak{z}^n_T)<\infty
$
for the function
$
\psi(x)=x\log(x)+1-x, \ x\ge 0.
$
We show first

\begin{equation*}
\widetilde{\E}^n_T\log(\mathfrak{z}
^n_T)<\infty
\end{equation*}
and, then, use an obvious identity
$
\E\psi(\mathfrak{z}^n_T)=\widetilde{\E}^n_T\log(\mathfrak{z}^n_T).
$

\begin{lemma}\label{lem-2.3}
$$
\widetilde{\E}^n_T\log(\mathfrak{z}
^n_T)\le\widetilde{\E}^n_T\int_0^T\int_{\mathbb{R}_+}I_{\{\sigma_n\ge
s\}}z^2\alpha^2(\omega,s)K(dz)ds\le\mathbf{r}
\Big[1+\widetilde{\E}^n_T\sup_{s\in[0,T\wedge\sigma_n]}(M^n_{s-})^2\Big].
$$
\end{lemma}
\begin{proof}

It is well known (see e.g. \cite{LSMar}, Ch. 2, \S4), the
Doleans-Dade equation \eqref{eq:sigmanz} obeys the unique solution
\begin{align*}
&\mathfrak{z}^n_t=\exp\Big(\int_0^t\int_0^tI_{\{\sigma_n\ge
s\}}\alpha(\omega,s)dM^n_s
\\
&\quad+\sum_{s\in[0,t\wedge\sigma_n]}\log
\Big\{\big[1+\alpha(\omega,s)\triangle M^n_s\big]
-\alpha(\omega,s)\triangle M^n_s\Big\}\Big).
\end{align*}
Recall
$
\alpha(\omega,s)\triangle M^n_s\ge 0.
$
Then
$
\log[1+\alpha(\omega,s)\triangle M^n_s]-\alpha(\omega,s)\triangle M^N_s\le 0
$
and, therefore,
$
\log\big(\mathfrak{z}^n_t)\le\int_0^tI_{\{\sigma_n\ge
s\}}\alpha(\omega,s)dM^n_s.
$
So, by Lemma \ref{lem-2.4z0},
\begin{equation*}
\log\big(\mathfrak{z}^n_T)\le\int_0^TI_{\{\sigma_n\ge
s\}}\alpha(\omega,s)dA^n_s+\int_0^TI_{\{\sigma_n\ge
s\}}\alpha(\omega,s)d\widetilde{M}^n_s.
\end{equation*}
The process $\widetilde{M}^n_t$ is square integrable martingale
with $\langle \widetilde{M}^n\rangle_T$ defined in \eqref{eq:uzhe}.
Therefore, the It\^o integral
$
\int_0^tI_{\{\sigma_n\ge
s\}}\alpha(\omega,s)d\widetilde{M}^n_s
$
is also the square integrable martingale with the quadratic variation
$$
QV:=\int_0^T\int_{\mathbb{R}_+}z^2I_{\{\sigma_n\ge s\}}\alpha^2(\omega,s)
\big[1+\alpha(\omega,s)z]K(dz)ds.
$$
Let us show that QV is bounded by a constant depending on $n$. Since $\int_{\mathbb{R}_+}(z^2\vee z^3)K(dz)\le \mathbf{r}$,
$$
QV\le \mathbf{r}\int_0^TI_{\{\sigma_n\ge s\}}[1+\alpha^3(\omega,s)]ds
\le \mathbf{r}T+\mathbf{r}\int_0^TI_{\{\sigma_n\ge s\}}\alpha^3(\omega,s)ds.
$$
Further,
$$
I_{\{\sigma_n\ge s\}}\alpha^3(\omega,s)=
[I_{\{\sigma_n\ge s\}}\alpha^2(\omega,s)]^{3/2}
\le\mathbf{r}\Big[1+\sup_{s'\in[0,s\wedge\sigma_n]}M^2_{s'-}\Big]^{3/2}\le \mathbf{r}n^{3/2},
$$
that is, $QV\le\mathbf{r}T[1+n^{3/2}]$.

Hence, $\widetilde{\E}^n_T\int_0^TI_{\{\sigma_n\ge
s\}}\alpha(\omega,s)d\widetilde{M}^n_s=0$ and
$
\widetilde{\E}^n_T\log\big(\mathfrak{z}^n_T)\le\widetilde{\E}^n\int_0^TI_{\{\sigma_n\ge
s\}}\alpha(\omega,s)dA^n_s.
$
So, it remains to recall the formula of $A^n_t$ (see Lemma \ref{lem-2.4z0}),
and $\int_{\mathbb{R}_+}z^2K(dz)$, and
$$
I_{\{\sigma_n\ge s\}}\alpha^2(\omega,s)=
I_{\{\sigma_n\ge s\}}\alpha^2(\omega,s)
\le\mathbf{r}\Big[1+\sup_{s'\in[0,s\wedge\sigma_n]}M^2_{s'-}\Big].
$$
\end{proof}

\subsection{Final step of the proof}
Now, we are in the position to compute $
\widetilde{\E}^n_T\sup\limits_{s\in[0,T]}|M^n_s|^2. $ The use of
$M^n_t=A^n_t+\widetilde{M}^n_t$ implies $
\widetilde{\E}^n_T\sup\limits_{t'\in[0,t]}|M^n_{t'}|^2 \le 2
\widetilde{\E}^n_T\sup\limits_{t'\in[0,t]}|A^n_{t'}|^2
+2\widetilde{\E}^n_T\sup\limits_{t'\in[0,t]}|\widetilde{M}^n_{t'}|^2.
$

In view of statement {\bf 3.} of Lemma \ref{lem-2.4z0}
$A^n_t=\int_{\mathbb{R}_+}z^2K(dz)\int_0^tI_{\{\sigma_n\ge
s\}} \alpha(\omega,s)ds$. So, by applying the
Cauchy-Schwarz inequality we obtain
\begin{equation}\label{eq:KS}
\widetilde{\E}^n_T\sup_{t'\in[0,t]}|A^n_{t'}|^2\le\mathbf{r}\widetilde{\E}^n_T
\int_0^tI_{\{\sigma_n\ge
s\}} \alpha^2(\omega,s)ds.
\end{equation}
Further, by the Doob maximal inequality and \eqref{eq:uhUh} we find
that
\begin{gather}
\widetilde{\E}^n_T\sup_{t'\in[0,t]}|\widetilde{M}^n_{t'}|^2 \le
4\widetilde{\E}^n_T\langle \widetilde{M}^n\rangle_t
=4\widetilde{\E}^n_T \int_0^t\int_{\mathbb{R}_+}z^2I_{\{\sigma_n\ge
s\}} \big(1+\alpha(\omega,s)z)K(dz)ds
\nonumber\\
\le\mathbf{r} \widetilde{\E}^n_T \int_0^tI_{\{\sigma_n\ge s\}}
\big[1+\alpha(\omega,s)\big]ds\le \mathbf{r} \widetilde{\E}^n_T
\int_0^tI_{\{\sigma_n\ge s\}} \big[1+\alpha^2(\omega,s)\big]ds
\label{eq:Doob}
\end{gather}
Now, a combination of \eqref{eq:KS}, and \eqref{eq:Doob} provides:
for any $t\le T$,
\begin{equation*}
\widetilde{\E}^n_T\sup_{t'\in[0,t]}|M^n_{t'}|^2\le
\mathbf{r}\widetilde{\E}^n_T\int_0^t[1+\alpha^2(\omega,s)]ds \le
\mathbf{r}\widetilde{\E}^n_T\int_0^t\Big[1+\sup_{s'\in[0,s]}|M^n_{s'}|^2\Big]ds.
\end{equation*}
Hence, the function
$V^n_t:=\sup_{n}\widetilde{\E}^n_T\sup_{s\in[0,t]}|M^n_{s}|^2$
solves an integral inequality:
$$
``V^n_t\le \mathbf{r}\Big(1+\int_0^tV^n_sds\Big)\text{''}\Rightarrow
``V^n_T\le \mathbf{r}e^{\mathbf{r}T}\text{''}.
$$
\qed


\begin{thebibliography}{9}

\bibitem{Benes} Benes, V.E. (1971)  Existence of optimal stochastic control laws
\emph{SIAM J. of Control}, 9 , 446-475

\bibitem{Girs} Girsanov, I.V. (1960) On transforming a certan class of stochastic processes by absolutely continuous
substitution of measures. \emph{Theory Probab. Appl.} {\bf 5},
285-301.

\bibitem{JSn} Jacod J., Shiryaev A.N.:  Limit theorems for stochastic
processes. 2nd ed. Springer-Verlag, Berlin (2003)

\bibitem{Hits}  Hitsuda Masuyuki. (1968) Representation of Gaussian processes
equivalent to Wiener process. \emph{Osaka J. Math.} {\bf 5},
299-312.

\bibitem{142} Karatzas, I. and Shreve, S.E. (1991): Brownian Motion
and Stochastic Calculus. Springer-Verlag, New York Berlin
Heidelberg.


\bibitem{7} Kazamaki, N. (1977) On a problem of Girsanov.// T\^ohoku Math. J., 29 , p. 597-
600.

\bibitem{Kry} Krylov, N.V. (8 May 2009) A simple proof of a result of A. Novikov.
arXiv:math/020713v2 [math.PR]

\bibitem{LSMar}  Liptser, R.Sh., Shiryayev, A.N. (1989) \emph{Theory of
Martingales.} Kluwer Acad. Publ.

\bibitem{LSI}  Liptser, R. Sh. and Shiryaev, A. N. (2000).
\emph{ Statistics of Random Processes {\rm I}}, 2nd ed.,  Springer,
Berlin - New York.

\bibitem{12} Novikov,  A.A. (1979) On the conditions of the uniform integrability of
the continuous nonnegative martingales.// Theory of Probability and
its Applications, 24, No. 4, p. 821-825.

\bibitem{UZ} "Ust\"unel. A and Zakai, M. (2000) \emph{Transformation of measure on Wiener
space.} Springer, Berlin - New York.
\end{thebibliography}
\end{document}